\documentclass[12pt, reqno]{amsart}
\usepackage{amsmath, amsthm, amscd, amsfonts, amssymb, graphicx, color}
\usepackage[bookmarksnumbered, colorlinks, plainpages]{hyperref}

\newtheorem{theorem}{Theorem}[section]
\newtheorem{lemma}[theorem]{Lemma}

\newtheorem{corollary}[theorem]{Corollary}
\theoremstyle{definition}

\theoremstyle{remark}
\newtheorem{remark}[theorem]{Remark}
\numberwithin{equation}{section}

\begin{document}
\title[Quasisimilarity and compact perturbations]{Quasisimilarity and compact perturbations}
\author [S. Mecheri,]{Salah Mecheri}

\address{Department of Mathematics\\
Faculty of Science and Informatics\\
El Bachir Ibrahimi University, Bordj Bou Arreridj, Algeria.}
\email{salahmecheri20@gmail.com}
 \subjclass{47A30, 47B47, 47B20}
\keywords{Hyponormal operators, $(M,k)$- quasi-*-class $Q$ operators, $k$-quasi-*-class $%
\mathbb{A}$ operators.} \maketitle
\begin{abstract}
In this paper we show that quasisimilar $n$-tuples of tensor
products of $m$-isometric operators have the same
spectra, essential spectra and indices. The properties of single
Fredholm operators possess \cite{4} is related to an important
property which has a leading role on the theory of Fredholm
operators: Fredholm n-tuples of operators. It is well known that a
Fredholm operator of index zero can be perturbed by a compact
operator to an invertible operator. In \cite[Problem 3]{5} the
author asked if this property holds in several variables. R. Gelca
in \cite{10} gave an example showing that this perturbation
property fails in several variables. In this paper we give a
positive answer to this question in case of tensor products of
some classes of operators.
\end{abstract}

\section{Introduction}
Let $H$ be a separable infinite dimensional complex Hilbert space,
and let $B(H)$ denote the algebra of all bounded linear operators
on $H$. For any operator $A$ in $B(H)$ set, as usual, $|A|=
(A^{*}A)^{{1\over 2}}$ and
$$[A^{*};A] = A^{*}A-AA^{*} = |A|^{2}-|A^{*}|^{2}$$ (the self commutator of $A$), and consider the
following standard definitions: $A$ is normal if $A^{*}A =
AA^{*}$, hyponormal if $A^{*}A-AA^{*}\geq 0$, $p$-hyponormal if
$|A|^{2p} -|A^{*}|^{2p}\geq 0 (0 < p \leq 1)$, and $(p;
k)$-quasihyponormal if $A^{*k}((A^{*}A)^{p}-(AA^{*})^{p})A^{k}\geq
0 (0 < p \leq 1; k\in \mathbb{N})$. If $p = 1, k = 1$ and $p = k =
1$, then $A$ is $k$-quasihyponormal, $p$-quasihyponormal and
quasihyponormal respectively. $A$ is said to be normaloid if
$||A|| = r(A)$ (the spectral radius of $A$). Let $(N), (HN),
(Q(p)); (Q(p; k))$ and $(NL)$ denote the classes constituting of
normal, hyponormal, $p$-quasihyponormal, $(p; k)$-
quasihyponormal, and normaloid operators. These classes are
related by proper inclusion: $$(N) \subset (HN) \subset
(Q(p))\subset (Q(p; k)).$$ (see \cite{16, Mecheri44, Mecheri19}). A $(p;
k)$-quasihyponormal operator is an extension of hyponormal,
$p$-hyponormal, $p$-quasihyponormal and $k$- quasihyponormal. A
1-hyponormal operator is called hyponormal operator, which has
been studied by many authors and it is known that hyponormal
operators have many interesting properties similar to those of
normal operators (see\cite{30}). A. Aluthge, B.C.Gupta, A.C. Arora
and P.Arora introduced $p$-hyponormal, $p$-quasihyponormal and
$k$-quasihyponormal operators, respectively (see\cite{2, 3, 7}),
and now it is known that these operators have many interesting
properties (see\cite{15, 23}). It is obvious that $p$-hyponormal
operators are $q$-hyponormal for $0 < q\leq p$ by Lowner-Heinz's
inequality (See \cite{18}). But $(p; 1)$-quasihyponormal operators
are not always $(q; 1)$-quasihyponormal operators for $0 < q \leq
p$ (see\cite{26}). Also, it is obvious that $(p;
k)$-quasihyponormal operators are $(p; k + 1)$-quasihyponormal. In
this paper we show that quasisimilar $n$-tuples of tensor products
of $(p; k)$-quasihyponormal operators have the same spectra,
essential spectra and indices. The properties of single Fredholm
operators possess \cite{4} is related to an important property
which has a leading role on the theory of Fredholm operators:
Fredholm n-tuples of operators. It is well known that a Fredholm
operator of index zero can be perturbed by a compact operator to
an invertible operator. In \cite[Problem 3]{5} the author asked if
this property holds in several variables. R. Gelca in \cite{10}
gave an example showing that this perturbation property fails in
several variables. In this paper we give a positive answer to this
question in case of tensor products of some classes of operators
\section{Quasisimilarity}
We begin by the following notations, definitions and Lemmas which
will be used for the sequel. An operator $T\in B(H)$ is said to
have Bishop's property $(\beta)$ if $(T-z)f_{n}(z)rightarrow 0$
uniformly on every compact subset of $D$ for analytic functions
$f_{n}(z)$ on $D$, then $f_{n}(z)\rightarrow 0$ uniformly on every
compact subset of $D$ \cite{Mecheri11, Mecheri22}. T is said to
have the single valued extension property if there exists no
nonzero analytic function f such that $(T - z)f(z)\equiv 0$. It is
clear if $T$ has Bishop's property ($\beta$), then $T$ has the
single valued extension property \cite{Mecheri33}. In this case,
the local resolvent $\rho_{T}(x)$ of $x\in H$ denotes the maximal
open set on which there exists a unique analytic function $f(z)$
satisfying $(T-z)f(z)\equiv x$. The local spectrum $\sigma_{T}
(x)$ of $x\in H$ is defined by $\sigma_{T} (x) =
\mathbb{C}\setminus\rho_{T}(x)$ and $X_{T} (F) = \{x\in H :
\sigma_{T} (x)\subset F\}$ for a subset $F \subset \mathbb{C}$.
$T$ is said to have finite ascent if $\ker T^{m} = \ker T^{m+1}$
for some positive integer $m$, and finite descent if $ran T^{n} =
ran T^{n+1}$ for some positive integer $n$. Laursen (Proposition
1.8 of (\cite{17}) proved that if $T-\lambda$ has finite ascent
for all $\lambda\in \mathbb{C}$, then $T$ has the single valued
extension property.
\begin{lemma}\cite{24} Let $A\in B(H)$ be a $(p;
k)$-quasihyponormal operator. If $M$ is an invariant subspace for
$A$, then the restriction of $A$ to $M$ is a $(p;
k)$-quasihyponormal operator.
\end{lemma}
\begin{lemma} \cite{12} Let $A;B\in B(H)$. Then $A;B$ are $(p; k)$-quasihyponormal if
and only if $A \otimes B$ is $(p; k)$-quasihyponormal.
 \end{lemma}
\begin{lemma}\cite{24} Let $S\in B(H)$ be $(p; k$)-quasihyponormal, the closure of
range of $S^{k}$ not dense and $S = \left(\begin{array}{lll}
        S_{1} &S_{2}\\
          0&S_{3}\\
 \end{array}\right)~\mbox{on} ~H=\overline{R(S^{k})}\oplus N(T^{*}),$
 then $S_{1}$ is $p$-hyponormal, $S^{k}_{3}=\{0\}$ and $\sigma(S)=\sigma(S_{1})\cup\{0\}$.
 \end{lemma}
 In the following lemma we will show that a $(p; k)$-quasihyponormal operator has Bishop's property
$(\beta)$.

\begin{lemma} Let $T\in B(H)$ be $(p; k)$-quasihyponormal. Then $T$ has Bishop's
property $(\beta)$. Hence $T$ has the single valued extension
property.
\end{lemma}
\begin{proof} Let $(T-z)f_{n}(z)\rightarrow 0$ uniformly on every compact subset of
$D$ for analytic functions $f_{n}(z)$ on $D$. Then we can write $$
\left(\begin{array}{lll}
        T_{1}-z &T_{2}\\
          0&T_{3}-z\\
 \end{array}\right) \left(\begin{array}{lll}
        f_{n_{1}}(z) \\
        f_{n_{2}}(z)
 \end{array}\right)
 =\left(\begin{array}{lll}
       (T_{1}-z) f_{n_{1}}(z)+T_{2}f_{n_{2}}(z) \\
        (T_{3}-z)f_{n_{2}}(z)
 \end{array}\right)\rightarrow 0.$$
 Since $T_{3}$ is nilpotent, $T_{3}$ has
Bishop's property ($\beta$). Hence $f_{n_{2}}(z)\rightarrow 0$
uniformly on every compact subset of D. Then ($T_{1} -
z)f_{n_{1}}(z)\rightarrow 0$. Since $T_{1}$ is $p$-hyponormal,
$T_{1}$ has Bishop's property ($\beta$) by Kimura \cite{13}. Hence
$f_{n_{1}}(z)\rightarrow 0$ uniformly on every compact subset of
$D$. Thus $T$ has Bishop's property ($\beta$).
\end{proof}
Let $\mathbb{T} = (T_{1}; . . . ; T_{n})\in B^{n}(H)$ be a
commuting $n$-tuple of operators in $B(H)$. Recall that
(cite{4,8}), $\mathbb{T}$ is said to be invertible if the Koszul
complex for $\mathbb{T}$, denoted $K(\mathbb{T};H)$ is exact at
every stage. Also, $\mathbb{T}$ is said to be Fredholm if the
Koszul complex $K(\mathbb{T};H)$ is Fredholm, i.e., all homologies
of $K(\mathbb{T};H)$ are finite dimensional. In this case the
index of $\mathbb{T}$, denoted $ind T$, is defined as the Euler
characteristic of $K(\mathbb{T};H)$. If $\mathbb{T}\in B^{n}(H)$
is Fredholm with index zero, then we say that $\mathbb{T}$ is
Weyl. We shall write $\sigma_{T} (T), \sigma_{T_{e}}(T)$, and
$\sigma_{T_{w}}(T)$ for the Taylor spectrum, the Taylor essential
spectrum and Taylor-Weyl spectrum of $T$, respectively, where
$$\sigma_{} (\mathbb{T}) = \{\lambda= (\lambda_{1}; . . . ;
\lambda_{n})\in \mathbb{C}^{n} :\, \mathbb{T}- \lambda\, \mbox{is
not invertible}\};$$
$$\sigma_{T_{e}}(T) = \{\lambda=
(\lambda_{1}; . . . ; \lambda_{n})\in \mathbb{C}^{n} :\,
\mathbb{T} - \lambda\, \mbox{is not Fredholmg};$$ and
$$\sigma_{T_{w}}(T) = \{\lambda= (\lambda_{1}; . . . ;
\lambda_{n})\in \mathbb{C}^{n} :\, \mathbb{T} - \lambda\, \mbox{is
not Weyl}.$$ For any polydisk $D\in \mathbb{C}^{n}$, let
$\mathbb{O}$ denote the Fr\'echet space of $H$-valued analytic
functions on $D$. Then we say ([7]) that a commuting $n$-tuple
$\mathbb{T}$ has the single extension property, shortened to SVEP,
if the Koszul complex $K((\mathbb{T} - \lambda);\mathbb{O}(D;H))$
is exact in positive degrees and $\mathbb{T}$ has the Bishop's
property ($\beta$) if it has SVEP and its Koszul complex has also
separated homology in degree zero. Obviously, the following
implication holds:
$$\mbox{Bishop's property}\,(\beta) )\Rightarrow\, \mbox{SVEP}.$$
Recall \cite{4} that $\lambda = \{\lambda_{1}; . . . ;
\lambda_{n}\} \in \mathbb{C}^{n}$ is said to be an eigenvalue of
$\mathbb{T}$ if there exists a non-zero vector $x\in H$ such that
$x \in \cap ker(T_{i} - \lambda_{i})$. We denote the set of all
eigenvalues of $\mathbb{T}$ by $\sigma_{p}(\mathbb{T})$ and
$$p_{00}(\mathbb{T}) = \sigma_{T}(\mathbb{T})\setminus \{\sigma_{T_{e}}(\mathbb{T})\cup
acc\sigma(\mathbb{T})\}$$ for he Riesz points of
$\sigma_{T}(\mathbb{T})$. Let $\mathcal{H} =
\overbrace{H_{1}\otimes . . .\otimes H_{n}}$ denote the completion
of $H_{1}\otimes . . .\otimes H_{n}$ with respect to some cross
norm, where $H_{i}\,(1 \leq i \leq n)$ are complex infinite
dimensional Hilbert spaces. Now we are ready to present our first
result.

\begin{theorem} Let $A_{i}; B_{i}\in B(H_{i})$ has Bishop's property $(\beta)$, let $T = (T_{1};
. . . ; T_{n})$ and $S = (S_{1}; . . . ; S_{n})$ be $n$-tuples of
operators $$T_{i} = I_{1}\otimes . . . \otimes I_{i-1}\otimes
A_{i}\otimes ...\otimes I_{n}$$ and $$S_{i} = I_{1}\otimes . . .
\otimes I_{i-1}\otimes B_{i}\otimes ...\otimes I_{n},$$
respectively. Assume that if $A_{i}$ has Bishop's property
($\beta$), then each $T_{i}$ has Bishop's property $\beta$. If $T
= (T_{1}; . . . ; T_{n})$ and $S = (S_{1}; . . . ; S_{n})$ are
quasisimilar $n$-tuples, then they have the same spectra,
essential spectra and indices, respectively.
\end{theorem}
\begin{proof}
By the fact that if $A_{i}$ has Bishop's property ($\beta$), then
each $T_{i}$ has Bishop's property ($\beta$). It follows from [28,
Corollary 2.2] that the $n$-tuple $T = (T_{1}; . . . ; T_{n}$) has
also Bishop's property ($\beta$). By the same arguments we prove
that $S = (S_{1}; . .. ; S_{n})$ has Bishop's property ($\beta$),
too. Recall \cite[Theorem 1; Corollary 1]{21} that if $T$ and $S$
are quasisimilar commuting $n$-tuples having Bishop's property
($\beta$), then these $n$-tuples have the same spectra, essential
spectra and indices, respectively. This completes the proof.
\end{proof}
\begin{lemma} Let $A_{i}\in B(H_{i})$ be $(p; k)$-quasihyponormal operators. Then each
$T_{i}$ has Bishop's condition ($\beta$).
\end{lemma}
\begin{proof}
Since $A_{i}$ are $(p; k)$-quasihyponormal operators, by applying
Lemma 2.2 it follows from a finite induction argument that $A_{i}$
are $(p; k)$- quasihyponormal operators if and only if $T_{i}$ are
$(p; k)$-quasihyponormal operators. Thus the fact that $A_{i}$
have Bishop's property ($\beta$), by Lemma 2.4. This implies that
each $T_{i}$ has Bishop's condition ($\beta$).
\end{proof}

\begin{corollary} Let $A_{i}; B_{i}\in B(H_{i})$ be $(p; k)$-quasihyponormals, let $\mathbb{T} = (T_{1};
. . . ; T_{n})$ and $\mathbb{S} = (S_{1}; . . . ; Sn)$ be
$n$-tuples of operators $T_{i} = I_{1}\otimes . . .\otimes
I_{i-1}\otimes A_{i} . . . I_{n}$ and $S_{i} = I_{1}\otimes . .
.\otimes I_{i-1}\otimes B_{i} . . . I_{n}$, respectively. If
$\mathbb{T} = (T_{1}; . . . ; T_{n})$ and $S = (S_{1}; . . . ;
S_{n})$ are quasisimilar $n$-tuples, then they have the same
spectra, essential spectra and indices, respectively.
\end{corollary}
\begin{proof} It suffices to apply Lemma 2.6 and Theorem 2.5.
\end{proof}
\section{Compact Perturbations}
We will define an interesting class of operators which contain the
classes of hyponormal and $p$-hyponrmal operators. An operator$
A\in B(H)$ is said to be paranormal if $||Ax||^{2} \leq
||A^{2}x||||x||$; for all $x\in H$. We have $$\mbox{hyponormal}
\subset p-\mbox{hyponormal} \subset
\mbox{paranormal}\subset\mbox{normaloid}.$$ A is said to be
$\log$-hyponormal if $A$ is invertible and satisfies the following
equality $\log(A^{*}A)\geq log(AA^{*})$: It is known that
invertible $p$-hyponormal operators are $\log$-hyponormal
operators but the converse is not true \cite{23}. However it is
very interesting that we may regards $\log$-hyponormal operator
are 0-hyponormal operator \cite{23}. The idea of $\log$-hyponormal
operator is due to Ando \cite{1}. For properties of
$\log$-hyponormal operators (see \cite{2, 23}). We say that an
operator $A\in B(H)$ belongs to the class $A$ if $|A^{2}| \geq
|A|^{2}$. Class A was first introduced by Furuta-Ito-Yamazaki
\cite{9} as a subclass of paranormal operators which includes the
classes of $p$-hyponormal and $\log$-hyponormal operators. The
following theorem is one of the results associated with class $A$.

\begin{theorem} \cite{9}

(1) Every $\log$-hyponormal operator is a class $A$ operator.

(2) Every class $A$ operator is a paranormal operator.

\end{theorem}
By the same argument used in the proof of Lemma 2.2, we can
prove the following lemma.

\begin{lemma} Let $A;B\in B(H)$. Then $A, B$ are class $A$ operators if and
only $A\otimes B$ is a class $A$ operator, too. Let $A_{i}\in
B(H_{i})$ and let $D(A_{i})$ be the set of operators $A_{i}$
satisfying the following properties

1) Every operator in $D(A_{i})$ is normaloid.

2) The restriction of $A_{i}\in D(Ai)$ to an invariant subspace
$M$ remains in $D(A_{i})$.

3) If $A_{i}|_{M}$ are normal operators, then $M$ is a reducing
subspace for $A_{i}$.
\end{lemma}
Now we are ready to give a positive answer to \cite[Problem 3]{5}.
\begin{theorem} Let $A_{i} \in D(A_{i})$ and let $\mathbb{T} = (T_{}; . . . ; T_{n})$ be
$n$-tuple of operators $T_{i} = I_{1}\otimes . . . \otimes
I_{i-1}\otimes A_{i}\otimes . . .\otimes I_{n}$ on $H_{i}$. If
$\mathbb{T}$ is Weyl but not invertible, then there exists an
invertible commuting $n$-tuple $\mathbb{S} = (S_{1}; . . . ;
S_{n})$ such that $\mathbb{T} = \mathbb{S} + \mathbb{F}$ for some
$n$-tuple of compact operators $F_{i} (i = 1; ... ; n)$ such that
$\mathbb{F}= (F_{1}, ..., F_{n})$.
\end{theorem}

\begin{proof} Since $\mathbb{T}$ is Weyl but not invertible, \cite[Theorem 1]{20}, implies
that $0 \in p_{00}(\mathbb{T})$. Let $f$ be the characteristic of
$0 \in iso\sigma_{T} (\mathbb{T})$. Since $f$ is analytic in a
neighborhood of $\sigma_{T}(\mathbb{T})$, \cite[Theorem 4.8;
Corollary 4.9]{25}, implies the existence of an idempotent $P_{0}
= f(\mathbb{T})\in B(H)$ such that $P_{0}T_{i} = T_{i}P_{0}$,
where $T_{i}$ is quasinilpotent on $ran P_{0}$ and $$ 0\notin
\sigma_{T} (\mathbb{T}_{ker P_{0}})$$. Since $A_{i}\in D(A_{i})$,
it follows that $A_{i}$ are normaloid and the restriction of each
$A_{i}$ to an invariant subspace remains in $D(A_{}$. Thus
$$T_{i}|_{ran P_{0}} = 0.$$ Since $T_{i}|_{ran P_{0}}=0$ and
$A_{i}\in D(A_{i})$, we have
$$\ker P_{0} = (ran P_{0})^{\perp},$$ that is, $P_{0}$ is an orthogonal
projection, and $$T_{i} = 0 \oplus T^{'}_{i} \mbox{on}\, H = ran
P_{0} \oplus ran P^{\perp}_{0}.$$ Since $0\in p_{00}(T)$, the
subspace $ran P_{0}$ is finite dimensional. Thus $P_{0}$ is a
compact operator on $H$. If we let $$F = (P_{i}; . . . ; P_{n})$$
and $$\mathbb{S} = \mathbb{T}-\mathbb{F} = (T_{1}-P_{0}; . . . ;
T_{n}-P_{0}),$$ then $S$ is a commuting $n$-tuples. Hence [4,
p.39], implies that $$\sigma_{T} (\mathbb{S}) = \sigma_{T}
((\mathbb{T} - \mathbb{F})|_{ran P_{0}})\cup \sigma_{T}
((\mathbb{T}- \mathbb{F})|_{ran P^{\perp}_{0}}).$$ Now it is clear
that $$0\notin \sigma_{T}((\mathbb{T}- \mathbb{F})|_{ran
P_{0}}).$$ This by (1) implies that $$0 \notin
\sigma_{T}((\mathbb{T} - \mathbb{F})|_{ran P^{\perp}_{0}}) =
\sigma_{T} (\mathbb{T}|_{\ker P_{0}}).$$ Hence $$0\notin
\sigma_{T} (\mathbb{S}),$$ that is, $\mathbb{S} =
\mathbb{T}-\mathbb{F}$ is invertible and $\mathbb{T} = \mathbb{S}
+\mathbb{F}$. This completes the proof.
\end{proof}
As application of the previous theorem we present the following
corollary. \begin{corollary} Let $A_{i}$ be class $A$ operators
and let $\mathbb{T} = (T_{1}; . . . ; T_{n})$ be $n$-tuple of
operators $$T_{i} = I_{1}\otimes . . .\otimes I_{i-1}\otimes
A_{i}\otimes . . .\otimes I_{n} \mbox{on}\, H.$$ If $\mathbb{T}$
is Weyl but not invertible, then there exists an invertible
commuting $n$-tuple $\mathbb{S} = (S_{1}; . . . ; S_{n})$ such
that $\mathbb{T} = \mathbb{S} + \mathbb{F}$ for some $n$-tuple of
compact operators $F_{i} (i = 1; . . . ; n)$, $\mathbb{F} =
(F_{1}; ...; F_{n})$.
\end{corollary}
\begin{proof} It is known \cite{27} that class $A$ operators are in $D(A_{i})$. Hence
it suffices to apply the previous theorem.
\end{proof}
\begin{remark} In Theorem 3.3. if $A_{i}$ was injective $\mathbb{T}$ must be
invertible. Indeed, by known results of Curto, Eschmeier and
Fialkow, the spectrum of the tensor product tuple $\mathbb{T}$ is
the cartesian product of the spectra of the operators $A_{i}$. The
essential spectrum of $\mathbb{T}$ consists of all points $z$ in
$\sigma(\mathbb{T})$ with $z_{i}\in  \sigma_{e}(A_{i})$ for at
least one index $i\in\{ 1; ...; n\}$. Since $$ind(\mathbb{T}) =
\prod ind(A_{i}),$$ the hypotheses of Theorem 3.2 imply that at
least one of the operators $A_{i}$ is a one to one Fredholm
operator with index zero. This forces the operator $A_{i}$ and
hence also the tuple $\mathbb{T}$ to be invertible.
\end{remark}
\textbf{Competing interests}. The authors declare that they have no competing interests.\\
\textbf{Availability of data and materials}
Data sharing not applicable to this paper as no data sets were generated or analyzed
during the current study.

{}

\begin{thebibliography}{99}
\bibitem{1} T.Ando, operators with a norm condition, Acta. Sci.
Math(Szeged), 33(1972), 169-178.

\bibitem{2} Aluthge, On p-hyponormal operators for 0 < p < 1, Integr.
Equat. Oper. Theor, 13(1990), 307-315.

\bibitem{3} S.C. Arora and P. Arora, On p-quqsihyponormal operators for 0 < p < 1, Yokohma
Math. J, 41(1993), 25-29.

\bibitem{4} R.Curto, Applications of several complex variables to
multiparameter spectral theory, Surveys in some recent results in
operator theory, J.B.Conway and B.Morel, eds. Vol II, Pitnam. Res.
Notes in Math. Ser. 192, Longman Publ., Co, London, 1988, 25-90.
\bibitem{5} R.Curto, Problems in multivariable operator theory,
Comtemp.Math, 120, Amer.Math.Soc. Providence, R.I, 1991.

\bibitem{6} Subnormal operators, Research notes in Math, Putnam Advanced
Publication Program, 51(1981).

\bibitem{7} S.L. Campel and B.C.Gupta, On k-quasihyponormal operators,
Math. Japn, 23(1978), 185-189.

\bibitem{8} J. Eschmeier and M.Putnar, spectral decomposition and analytic
sheaves, Ox- ford University Press, Oxford, 1996.

\bibitem{9} T.Furuta, M. Ito and T. Yamazaki, A subclass of paranormal
operators includ- ing the class of log-hyponormal and several
related classes, Scientiae Math. Japn, 1(1998), 389-403.

\bibitem{10} R.Gelca, Compact perturbation of Fredhol n-tuples, Proc. Amer.
Math. Soc, 22(1994), 195-198.

\bibitem{11} I.H. Jeon and K.Tanahashi and A. Uchiyama, On qasisimilarity
for log- hyponormal operator, Glasgow Math. J, 46(2004), 169-176.
\bibitem{12} I.H. Kim, On (p; k)-quasihyponormal operators, Math. Ineq.
Appl, 7(2004), 629-638.

\bibitem{13} F.Kimura, Analysis of non-normal operator via Aluthge
transformation, In- tegr. Equat. Operat. Theor, 50(2004), 375-384.
\bibitem{14} M.Y. Lee and S.H. Lee, An extension of the Fuglede-Putnam's
theorem to p-quasihyponormal operators, Bull. Korean Math. Soc,
35(1998), 319-324.

\bibitem{15} M.Y.Lee, An extension of the Fuglede-Putnam's theorem to (p;
k)- quasihyponormal operators, Kyungpook Math. J, 44(2004),
593-596.

\bibitem{16} M.Y. Lee and S.H. Lee, Some generalized
theorems on p-quasihyponormal operators, 0 < p < 1,
Nihonkai.Math.J, 8(1997), 109-115.

\bibitem{17} K.B. Laursen, Operators with finite ascent, Pacific.J.Math,
152(1992).

\bibitem{18} K. Lowner, Uber monotone Matrixfunktionen, Math. Z, 38(1934),
177-216.

\bibitem{19} S.Mecheri, and A.Uchiyama, An extension of Fuglede Putnam-theorem to class A
operators, Math. Ineq. Appl, 13(2010), 57-61.
\bibitem{Mecheri11} S.Mecheri, Bishop's property, hypercyclicity and hyperinvariant subspace, Operator and Matrices, 8(2014), 555-562.
\bibitem{Mecheri22} S. Mecheri, Bishop's property, SVEP and Dunford property (C), Elect. Jour. Lin. Alg, 23 (2012), 523-529.
\bibitem{Mecheri33} S. Mecheri, On the normality of operators, Revista Colombiana de Matemáticas, 39(2005),87-95
\bibitem{Mecheri44} S. Mecheri, A. Mansour, On the operator equation $AXB-XD= E$, Lobachevskii Journal of Mathematics, 30(2009), 224-228.
\bibitem{Mecheri55} S. Mecheri, Weyl Type theorem for posinormal operator, Mathematical Proceedings of the Royal Irish Academy, 108(2008), 69-79.
\bibitem{20} M.Putinar, On Weyl's spectrum on several variables,
Math.Japonica, 50(1999), 355-357

\bibitem{21} M.Putinar, Quasisimilarity of tuples with Bishop's property
($\beta$), Integral equa- tions and operator theory, 15(1992),
1047-1052.

\bibitem{22} J. Snadar, Bishop's condition ($\beta$), Glasg. Math. J, 26(1985),
35-46.

\bibitem{23} K. Tanahashi and A. Uchiyama, Isolated points of spectrum of
p- quasihyponormal operators, Lin. alg. Appl, 382(2004), 221-229.

\bibitem{24} K.Tanahashi, A.Uchiyama and M. Cho, Isolated points of
spectrum of $(p; k)$- quasihyponormal operators, Lin. Alg. Appl,
382(2004),221-229.

\bibitem{25} J.L.Taylor, The analytic functional mathcalculus for several
commuting oper- ators, Acta Sci.Math, 125(1970), 1-38.

\bibitem{26} A.Uchiyama, An example of a p-quasihyponormal operator,
Yokohama.Math.J 46(1999)179-180.

\bibitem{27} A. Uchiyama, Weyl's theorem for class A operators, Math. Ineq.
Appl, 4(2001), 143-150

\bibitem{28} R. Wolff, Bishop's property $\beta$ for tensor product tuples of
operators, J.Operator Theory, 42(1999), 371-377.

\bibitem{29} Y.J.Woo, Quasi-similarity and injective p-quasihyponormal
operators, Bull Korean.Math.Soc, 42(2005),653-659.
\bibitem{30} D.Xia, Spectral theory of hyponormal operators, Birkhauser Verlag,
Boston, London, 1983.
\end{thebibliography}
\end{document}